\documentclass[12pt]{amsart}
\usepackage{amsmath,amssymb,amsfonts,amsthm,amscd,indentfirst}
\usepackage{amsmath,amssymb,amsfonts,amsthm,amscd}
\usepackage{indentfirst}
\usepackage{hyperref}

\numberwithin{equation}{section}

\newtheorem{prop}{Proposition}[section]
\newtheorem{theo}[prop]{Theorem}
\newtheorem{lemm}[prop]{Lemma}

\def\and{\quad{\rm and}\quad}

\def\<{\langle}
\def\>{\rangle}

\usepackage{graphicx}

\begin{document}
\title[Interior $C^2$ estimate for Hessian quotient equation in dimension three]{Interior $C^2$ estimate for Hessian quotient equation in dimension three}
\author[Siyuan Lu]{Siyuan Lu}
\address{Department of Mathematics and Statistics, McMaster University, 
1280 Main Street West, Hamilton, ON, L8S 4K1, Canada.}
\email{siyuan.lu@mcmaster.ca}
\thanks{Research of the author was supported in part by NSERC Discovery Grant.}

\begin{abstract}
In this paper, we establish an interior $C^2$ estimate for the Hessian quotient equation $\left(\frac{\sigma_3}{\sigma_1}\right)(D^2u)=f$ in dimension three. A crucial ingredient in our proof is a Jacobi inequality.
\end{abstract}

\maketitle 

\section{Introduction}

In the study of fully nonlinear equations, apriori estimates play an essential role. Among them, interior estimates are of particular interest due to their elegant form. By interior estimates, we refer to the pure interior estimates which are independent of any information of the solution on the boundary.

\medskip

The study of interior $C^2$ estimate for fully nonlinear equations dates back to Heinz's work on Weyl's embedding problem. In \cite{H}, Heinz established an interior $C^2$ estimate for the Monge-Amp\`ere equation $\det(D^2u)=f$ for $n=2$, see also recent proofs by Chen, Han and Ou \cite{CHO} and Liu \cite{Liu}. However, interior $C^2$ estimate fails for the Monge-Amp\`ere equation for $n\geq 3$, due to the counter example by Pogorelov \cite{Pb}. Pogorelov's example was extended by Urbas \cite{Urbas} to show that interior $C^2$ estimate fails for the Hessian equation $\sigma_k(D^2u)=f$ for $k\geq 3$. Here $\sigma_k$ is the $k$-th elementary symmetric function. The remaining case for the Hessian equation is $\sigma_2(D^2u)=f$ and it is a longstanding problem.

For $\sigma_2$ equation, a major breakthrough was made by Warren and Yuan. In \cite{WY09}, they obtained an interior $C^2$ estimate for $\sigma_2(D^2u)=1$ for $n=3$. For general right hand side $f$ in dimension $3$, interior $C^2$ estimate was proved recently by Qiu \cite{Q1,Q2}. For $n=4$, interior $C^2$ estimate for $\sigma_2(D^2u)=1$ was established in a cutting edge paper by Shankar and Yuan \cite{SY-23}. For general $n$, interior $C^2$ estimate for $\sigma_2(D^2u)=1$ was proved by McGonagle, Song and Yuan \cite{MSY} under the condition $D^2u> -c(n)$, by Shankar and Yuan \cite{SY-20} under the condition $D^2u>-K$, by Shankar and Yuan \cite{SY-23} under the condition $D^2u\geq -c(n)\Delta u$ and by Mooney \cite{Mooney} under the condition $D^2u>0$. For general $n$ and general right hand side $f$, interior $C^2$ estimate was established by Guan and Qiu \cite{GQ} under the condition $\sigma_3(D^2u)>-A$.

\medskip

Apart from Hessian equation, another known fully nonlinear equation that admits interior $C^2$ estimate is the special Lagrangian equation $\sum_i\arctan(\lambda_i)=\Theta$ \cite{HL}, where $\lambda_i$'s are the eigenvalues of $D^2u$ and $\Theta$ is called the phase. 

The study of interior $C^2$ estimate for special Lagrangian equation was initiated in a pioneer work by Warren and Yuan \cite{WY09} mentioned above. In \cite{WY09}, they obtained an interior $C^2$ estimate for $\sum_i\arctan(\lambda_i)=\frac{\pi}{2}$ for $n=3$ (This is equivalent to $\sigma_2(D^2u)=1$ for $n=3$). Interior $C^2$ estimate for convex solutions of special Lagrangian equation was proved by Chen, Warren and Yuan \cite{CWY}. Interior $C^2$ estimate for special Lagrangian equation with critical and supercritical phase $|\Theta|\geq \frac{(n-2)\pi}{n}$ was established via works of Wang, Warren and Yuan \cite{WY09,WY10,WdY14}. The notion of critical phase was introduced by Yuan \cite{Y06} and it illustrates the dramatic difference between subcritical phase and supercritical phase. For variable phase functions $\theta(x)$, interior $C^2$ estimate was obtained by Bhattacharya and Shankar \cite{BS1,BS2} for convex solutions, by Bhattacharya \cite{B1} for supercritical phase $|\theta(x)|>\frac{(n-2)\pi}{n}$ and by the author \cite{Lu} for critical and supercritical phase $|\theta(x)|\geq\frac{(n-2)\pi}{n}$. Note that interior $C^2$ estimate fails for special Lagrangian equation with subcritical phase by examples of Nadirashvili and Vlăduţ \cite{NV} and Wang and Yuan \cite{WdY13}. This illustrates the subtlety of the interior $C^2$ estimate.

\medskip

Besides Hessian equation and special Lagrangian equation, another important class of fully nonlinear equation is the Hessian quotient equation. In contrast, very little was known concerning the interior $C^2$ estimate for Hessian quotient equation. To our best knowledge, the only known case is $\left(\frac{\sigma_3}{\sigma_1}\right)(D^2u)=1$ for $n=3$ and $4$ proved by Chen, Warren and Yuan \cite{CWY} and Wang and Yuan \cite{WdY14} mentioned above. In this case, the equation is in fact the special Lagrangian equation with phase $\pi$, and the special Lagrangian structure plays an essential role in their proofs. For general Hessian quotient equation, there is no special Lagrangian structure. It is then of great interest to know whether interior $C^2$ estimate holds for general Hessian quotient equation with general right hand side. 

\medskip

The main result of this paper is an interior $C^2$ estimate for the Hessian quotient equation $\left(\frac{\sigma_3}{\sigma_1}\right)(D^2u)=f$ in dimension three.
\begin{theo}\label{Theorem}
Let $f\in C^{1,1}(B_{10})$ be a positive function and let $u\in C^4(B_{10})$ be a convex solution of 
\begin{align}\label{eq-orginal}
F(D^2u)=\left(\frac{\sigma_3}{\sigma_1}\right)(D^2u)=f(x),\quad \textit{in}\quad B_{10}\subset \mathbb{R}^3.
\end{align}
Then we have
\begin{align*}
|D^2u(0)|\leq C,
\end{align*}
where $C$ depends only on $\|u\|_{C^{0,1}(\overline{B}_9)}$, $\min_{\overline{B}_9}f $ and $\|f\|_{C^{1,1}(\overline{B}_9)}$. 
\end{theo}

We remark that our equation does not have special Lagrangian structure and the argument in \cite{CWY} does not work in our case.

As we can see from the discussions above, fully nonlinear equations that admit interior $C^2$ estimate are very rare, namely $\sigma_2$ equation and special Lagrangian equation with certain constraints. Our theorem provides a new example which does not belong to the above two classes of equations, that also admits interior $C^2$ estimate.

\medskip

The crucial ingredient of our proof is a Jacobi inequality for $\ln \lambda_1$, where $\lambda_1$ is the largest eigenvalue of $D^2u$. The main difficulty is to handle the third order terms. This turns out to be extremely delicate and subtle. Our key observation is that we can extract enough positive terms from $-\sum_{p,q,r,s} F^{pq,rs}u_{pq1}u_{rs1}$ to offset the negative term $-F^{11}u_{111}^2$. This is the place where we make full use of our equation. It is not clear to us whether such Jacobi inequality holds for general Hessian quotient equations.

After obtaining the Jacobi inequality, we adopt the idea by Shankar and Yuan \cite{SY-20} to perform a Legendre transform. The new equation is uniformly elliptic and we can bound $\ln\lambda_1$ by its integral. The last step is to use integration by parts to estimate this integral. We encounter new difficulty here due to the fact that $F^{ij}$ is not divergence free. To overcome this difficulty, we employ the new coefficient $H^{ij}=\sigma_1F^{ij}$. Although $H^{ij}$ not divergence free, we are able to control all terms generating from integration by parts. 

\medskip

The organization of the paper is as follows. In Section 2, we will collect some basic properties of the operator $F$ and a simple observation of our equation. In Section 3, we will establish the crucial Jacobi inequality. In Section 4, we will perform a Legendre transform to bound $\ln\lambda_1$ by its integral. In Section 5, we will complete the proof of Theorem \ref{Theorem} using integration by parts.

After we completed our paper, we were informed by Yuan that Zhou has independently proved the main theorem above in \cite{Zhou} using a completely different method.

\section{Preliminaries}

We first collect some basic properties of the operator $F$.
\begin{lemm}\label{F^{ijkl}}
Suppose $D^2u$ is diagonalized at $x_0$. Then at $x_0$, we have
\begin{align*}
F^{ij}=\frac{\partial F}{\partial u_{ij}}=\frac{\sigma_3^{ij}}{\sigma_1}-\frac{\sigma_3}{\sigma_1^2}\delta_{ij},
\end{align*}
\begin{align*}
F^{ij,kl}=\frac{\partial^2F}{\partial u_{ij}\partial u_{kl}}=\begin{cases}
-2\frac{\sigma_3^{ii}}{\sigma_1^2}+2\frac{\sigma_3}{\sigma_1^3},\quad &i=j=k=l,\\
\frac{\sigma_3^{ii,kk}}{\sigma_1}-\frac{\sigma_3^{ii}}{\sigma_1^2}-\frac{\sigma_3^{kk}}{\sigma_1^2}+2\frac{\sigma_3}{\sigma_1^3}, &i=j,k=l,i\neq k,\\
\frac{\sigma_3^{ij,ji}}{\sigma_1}, & i=l,j=k,i\neq j,\\
0, & \textit{otherwise}.
\end{cases}
\end{align*}

In particular, 
\begin{align*}
F^{ij,ji}=\frac{F^{ii}-F^{jj}}{\lambda_i-\lambda_j},\quad i\ne j,\quad\lambda_i\neq \lambda_j,
\end{align*}
where $\lambda_i$'s are the eigenvalues of $D^2u$. 

\end{lemm}

We also need the following simple observation of our equation.
\begin{lemm}\label{lambda_3}
Let $u$ be a convex solution of equation (\ref{eq-orginal}). Suppose $D^2u$ is diagonalized at $x_0$ such that $\lambda_1\geq \lambda_2\geq \lambda_3$. Then at $x_0$, we have
\begin{align*}
\lambda_3\leq \sqrt{3f},\quad \lambda_2\geq \sqrt{\frac{f}{3}}.
\end{align*}
\end{lemm}

\begin{proof}
We have
\begin{align*}
f=\frac{\sigma_3}{\sigma_1}=\frac{\lambda_1\lambda_2\lambda_3}{\sigma_1}\geq \frac{\lambda_2\lambda_3}{3}\geq \frac{\lambda_3^2}{3},
\end{align*}
i.e. $\lambda_3\leq \sqrt{3f}$.

It follows that
\begin{align*}
\lambda_2=\frac{f\sigma_1}{\lambda_1\lambda_3}\geq \frac{f}{\lambda_3}\geq \frac{f}{\sqrt{3f}}=\sqrt{\frac{f}{3}}.
\end{align*}

\end{proof}

\section{Jacobi inequality}
In this section, we will derive a Jacobi inequality for $\ln \lambda_1$, where $\lambda_1$ is the largest eigenvalue of $D^2u$. This is the crucial ingredient of our proof. 

Instead of equation (\ref{eq-orginal}), we will consider the following slightly more general equation
\begin{align}\label{eq}
F(D^2u)=\left(\frac{\sigma_3}{\sigma_1}\right)(D^2u)=f(x,u),\quad \textit{in}\quad B_{10}\subset\mathbb{R}^3.
\end{align}

\begin{lemm}\label{Jacobi}
Let $f\in C^{1,1}(B_{10}\times \mathbb{R})$ be a positive function and let $u\in C^4(B_{10})$ be a convex solution of (\ref{eq}). For any $x_0\in B_9$, suppose that $D^2u$ is diagonalized at $x_0$ such that $\lambda_1\geq \lambda_2\geq \lambda_3$ and $\lambda_1\geq \Lambda(f)$, where $\Lambda(f)$ is a large constant depending only on $\max_{\overline{B}_9\times [-M, M]}f$. Set $b=\ln\lambda_1$. Then at $x_0$, we have
\begin{align*}
\sum_i F^{ii}b_{ii}\geq \frac{1}{432}\sum_i F^{ii}b_i^2-C,
\end{align*}
in the viscosity sense, where $C$ depends only on $\|u\|_{C^{0,1}(\overline{B}_9)}$, $\min_{\overline{B}_9\times [-M,M]}f $ and $\|f\|_{C^{1,1}\left( \overline{B}_9\times [-M,M]\right)}$. Here $M$ is a large constant satisfying $\|u\|_{L^\infty(\overline{B}_9)}\leq M$.
\end{lemm}

We will prove Lemma \ref{Jacobi} depending on the multiplicity of $\lambda_1$. Note that by Lemma \ref{lambda_3}, $\lambda_3\leq \sqrt{3f}$. Therefore, as long as we choose $\Lambda(f)>\sqrt{3f}$, $\lambda_1$ will not have multiplicity $3$. In the following two subsections, we will prove Lemma \ref{Jacobi} for cases that $\lambda_1$ has multiplicity $1$ and $2$ respectively.

\subsection{$\lambda_1$ has multiplicity $1$}

In this subsection, we will prove Lemma \ref{Jacobi} for the case that $\lambda_1$ has multiplicity $1$, i.e. $\lambda_1>\lambda_2\geq \lambda_3$.

\begin{proof}
Differentiating $\lambda_1$, we have
\begin{align*}
{\lambda_1}_i=u_{11i},\quad {\lambda_1}_{ii}= u_{11ii}+2\sum_{p\neq 1}\frac{u_{1pi}^2}{\lambda_1 -\lambda_p}.
\end{align*}

It follows that
\begin{align*}
b_{ii}=\frac{{\lambda_1}_{ii}}{\lambda_1}-\frac{{\lambda_1}_i^2}{\lambda_1^2}= \frac{u_{11ii}}{ \lambda_1}+2\sum_{p\neq 1}\frac{u_{1pi}^2}{\lambda_1 (\lambda_1 -\lambda_p )}-\frac{ u_{11i}^2}{\lambda_{1}^2}.
\end{align*}

Contracting with $F^{ii}$, we have
\begin{align}\label{m-0}
\sum_i F^{ii}b_{ii}\geq \sum_i \frac{F^{ii}u_{11ii}}{ \lambda_1}+2\sum_i\sum_{p\neq 1}\frac{F^{ii}u_{1pi}^2}{\lambda_1 (\lambda_1 -\lambda_p )}-\sum_i \frac{F^{ii} u_{11i}^2}{\lambda_{1}^2}.
\end{align}

Differentiating (\ref{eq}), we have
\begin{align}\label{m-0-1}
\sum_i F^{ii}u_{ii11}+\sum_{p,q,r,s} F^{pq,rs}u_{pq1}u_{rs1}=f_{11}\geq -C\lambda_1-C,
\end{align}
where $C$ depends only on $\|u\|_{C^{0,1}(\overline{B}_9)}$ and $\|f\|_{C^{1,1}\left( \overline{B}_9\times [-M,M]\right)}$. Here $M$ is a large constant satisfying $\|u\|_{L^\infty(\overline{B}_9)}\leq M$. 

In the following, we will denote $C$ to be a constant depending only on $\|u\|_{C^{0,1}(\overline{B}_9)}$, $\min_{\overline{B}_9\times [-M,M]}f $ and $\|f\|_{C^{1,1}\left( \overline{B}_9\times [-M,M]\right)}$. It may change from line to line. 

Plugging (\ref{m-0-1}) into (\ref{m-0}), we have
\begin{align}\label{m-1}
\sum_i F^{ii}b_{ii}\geq -\sum_{p,q,r,s} \frac{F^{pq,rs}u_{pq1}u_{rs1}}{\lambda_1 }+2\sum_i\sum_{p\neq 1}\frac{F^{ii}u_{1pi}^2}{\lambda_1 (\lambda_1 -\lambda_p )}-\sum_i \frac{ F^{ii}u_{11i}^2}{\lambda_{1}^2}-C.
\end{align}

By Lemma \ref{F^{ijkl}}, we have
\begin{align}\label{m-0-2}
-\sum_{p,q,r,s} F^{pq,rs}u_{pq1}u_{rs1}=-\sum_{i,j}F^{ii,jj}u_{ii1}u_{jj1}-\sum_{i\neq j}F^{ij,ji}u_{ij1}^2.
\end{align}

Plugging into (\ref{m-1}), we have
\begin{align}\label{m-2}
\sum_i F^{ii}b_{ii}\geq &\ -\sum_{i, j} \frac{F^{ii,jj}u_{ii1}u_{jj1}}{\lambda_1} -\sum_{i\neq j} \frac{F^{ij,ji}u_{ij1}^2}{\lambda_1}+2\sum_i\sum_{p\neq 1}\frac{F^{ii}u_{1pi}^2}{\lambda_1 (\lambda_1 -\lambda_p )}\\\nonumber
&\ -\sum_i \frac{ F^{ii}u_{11i}^2}{\lambda_{1}^2}-C.
\end{align}

Since $\lambda_1>\lambda_2\geq \lambda_3$, by Lemma \ref{F^{ijkl}}, we have
\begin{align*}
-\sum_{i\neq j} \frac{F^{ij,ji}u_{ij1}^2}{\lambda_1}=\sum_{i\neq j} \frac{\sigma_3^{ii,jj} u_{ij1}^2}{\lambda_1\sigma_1}\geq 2\sum_{i\neq 1} \frac{\sigma_3^{11,ii} u_{11i}^2}{\lambda_1\sigma_1}=2\sum_{i\neq 1} \frac{(F^{ii}-F^{11})u_{11i}^2}{\lambda_1(\lambda_1-\lambda_i)}.
\end{align*}

On the other hand,
\begin{align*}
2\sum_i\sum_{p\neq 1}\frac{F^{ii}u_{1pi}^2}{\lambda_1 (\lambda_1 -\lambda_p )} \geq&\ 2\sum_{p\neq 1}\frac{F^{pp}u_{1pp}^2}{\lambda_1 (\lambda_1 -\lambda_p )}+ 2\sum_{p\neq 1}\frac{F^{11}u_{1p1}^2}{\lambda_1 (\lambda_1 -\lambda_p )}\\
=&\ 2\sum_{i\neq 1}\frac{F^{ii}u_{ii1}^2}{\lambda_1 (\lambda_1 -\lambda_i )}+2\sum_{i\neq 1}\frac{F^{11}u_{11i}^2}{\lambda_1 (\lambda_1 -\lambda_i )}.
\end{align*}

Combining the above two inequalities and plugging into (\ref{m-2}), we have
\begin{align}\label{m-3}
\sum_i F^{ii}b_{ii}\geq &\ -\sum_{i,j} \frac{F^{ii,jj} u_{ii1}u_{jj1}}{\lambda_1} +2\sum_{i\neq 1}\frac{F^{ii}u_{ii1}^2}{\lambda_1 (\lambda_1 -\lambda_i )}+2\sum_{i\neq 1}\frac{F^{ii}u_{11i}^2}{\lambda_1 (\lambda_1 -\lambda_i )}\\\nonumber
&\ -\sum_i \frac{ F^{ii}u_{11i}^2}{\lambda_{1}^2}-C\\\nonumber
\geq  &\ -\sum_{i, j} \frac{F^{ii,jj}u_{ii1}u_{jj1}}{\lambda_1} +2\sum_{i\neq 1}\frac{F^{ii}u_{ii1}^2}{\lambda_1 (\lambda_1 -\lambda_i )}+\sum_{i\neq 1}\frac{F^{ii}u_{11i}^2}{\lambda_1^2} \\\nonumber
&\ -\frac{ F^{11}u_{111}^2}{\lambda_{1}^2}-C.
\end{align}

By Lemma \ref{F^{ijkl}}, we have
\begin{align}\label{m-3-1}
F^{ii,ii}=-2\frac{\sigma_3^{ii}}{\sigma_1^2}+2\frac{\sigma_3}{\sigma_1^3}=-\frac{2}{\sigma_1}F^{ii}.
\end{align}

Plugging into (\ref{m-3}), we have
\begin{align}\label{m-4}
\sum_i F^{ii}b_{ii}\geq&\ -\sum_{i\neq j} \frac{F^{ii,jj}u_{ii1}u_{jj1}}{\lambda_1} +2\sum_{i\neq 1}\frac{F^{ii}u_{ii1}^2}{\lambda_1 (\lambda_1 -\lambda_i )}+\sum_{i\neq 1}\frac{F^{ii}u_{11i}^2}{\lambda_1^2}\\\nonumber
&\ -\frac{1}{\lambda_1}\left(\frac{1}{\lambda_1}-\frac{2}{\sigma_1}\right)F^{11}u_{111}^2-C.
\end{align}

By Lemma \ref{F^{ijkl}}, for $i\neq j$, we have
\begin{align}\label{m-4-1}
F^{ii,jj}=&\ \frac{\sigma_3^{ii,jj}}{\sigma_1}-\frac{\sigma_3^{ii}}{\sigma_1^2}-\frac{\sigma_3^{jj}}{\sigma_1^2}+2\frac{\sigma_3}{\sigma_1^3}\\\nonumber
=&\ \frac{\sigma_3}{\sigma_1}\left(\frac{1}{\lambda_i\lambda_j}-\frac{1}{\lambda_i\sigma_1}-\frac{1}{\lambda_j\sigma_1}+\frac{2}{\sigma_1^2}\right)\\\nonumber
=&\ f\bigg(\left(\frac{1}{\lambda_i}-\frac{1}{\sigma_1}\right) \left(\frac{1}{\lambda_j}-\frac{1}{\sigma_1}\right)+\frac{1}{\sigma_1^2}\bigg)\\\nonumber
=&\ \frac{F^{ii}F^{jj}}{f}+\frac{f}{\sigma_1^2}.
\end{align}
We have used the fact that $F^{ii}=f\left(\frac{1}{\lambda_i}-\frac{1}{\sigma_1}\right)$ in the last line.

Together with the fact that $\sum_i F^{ii}u_{ii1}=f_1$ and $F^{11}=f\left(\frac{1}{\lambda_1}-\frac{1}{\sigma_1}\right)$, we have
\begin{align*}
&\  -\sum_{i\neq j}\frac{F^{ii,jj}u_{ii1}u_{jj1}}{\lambda_1}\\
=&\ -\frac{1}{f\lambda_1} \sum_{i\neq j}F^{ii}F^{jj}u_{ii1}u_{jj1}-\frac{f}{\lambda_1 \sigma_1^2} \sum_{i\neq j}u_{ii1}u_{jj1}\\
=&\ -\frac{1}{f\lambda_1}\left(\sum_iF^{ii}u_{ii1}\right)^2+\frac{1}{f\lambda_1}\sum_i (F^{ii})^2u_{ii1}^2-\frac{f}{\lambda_1\sigma_1^2} \sum_{i\neq j}u_{ii1}u_{jj1}\\
=&\ -\frac{f_1^2}{f\lambda_1}+\sum_i \frac{(F^{ii})^2u_{ii1}^2}{f\lambda_1}-\frac{f}{\lambda_1\sigma_1^2} \sum_{i\neq j}u_{ii1}u_{jj1}\\
\geq &\ \sum_{i\neq 1}\frac{(F^{ii})^2u_{ii1}^2}{f\lambda_1}+\frac{1}{\lambda_1}\left(\frac{1}{\lambda_1}-\frac{1}{\sigma_1}\right)F^{11}u_{111}^2-\frac{f}{\lambda_1\sigma_1^2} \sum_{i\neq j}u_{ii1}u_{jj1}-C.
\end{align*}

Plugging into (\ref{m-4}), we have
\begin{align}\label{m-5}
\sum_i F^{ii}b_{ii}\geq&\ -\frac{f}{\lambda_1\sigma_1^2} \sum_{i\neq j}u_{ii1}u_{jj1}+\sum_{i\neq 1} \frac{(F^{ii})^2u_{ii1}^2}{f\lambda_1}+2\sum_{i\neq 1}\frac{F^{ii}u_{ii1}^2}{\lambda_1 (\lambda_1 -\lambda_i )}\\\nonumber
&\ +\sum_{i\neq 1}\frac{F^{ii}u_{11i}^2}{\lambda_1^2}+\frac{F^{11}u_{111}^2}{\lambda_1\sigma_1}-C.
\end{align}

\medskip

{\bf Case A:} $\lambda_2\leq \frac{\lambda_1}{4}$.

\medskip

Note that for each $i\neq 1$, we have
\begin{align}\label{m-5-1}
&\ -\frac{2f}{\lambda_1\sigma_1^2}u_{111}u_{ii1}+\frac{(F^{ii})^2u_{ii1}^2}{f\lambda_1}\\\nonumber
=&\ \left(\frac{F^{ii}u_{ii1}}{\sqrt{f\lambda_1}}-\frac{f\sqrt{f}}{\sqrt{\lambda_1}\sigma_1^2}\frac{u_{111}}{F^{ii}}\right)^2- \frac{f^3}{\lambda_1\sigma_1^4} \frac{u_{111}^2}{(F^{ii})^2}\\\nonumber
\geq&\ - \frac{f^3}{\lambda_1\sigma_1^4} \frac{u_{111}^2}{(F^{ii})^2}.
\end{align}

Consequently,
\begin{align*}
-\frac{f}{\lambda_1\sigma_1^2} \sum_{i\neq j}u_{ii1}u_{jj1}+\sum_{i\neq 1} \frac{(F^{ii})^2u_{ii1}^2}{f\lambda_1}\geq -\frac{2f}{\lambda_1\sigma_1^2} u_{221}u_{331}-\sum_{i\neq 1} \frac{f^3}{\lambda_1\sigma_1^4} \frac{u_{111}^2}{(F^{ii})^2}.
\end{align*}

Plugging into (\ref{m-5}), we have 
\begin{align}\label{m-6}
\sum_i F^{ii}b_{ii}\geq&\ -\frac{2f}{\lambda_1\sigma_1^2} u_{221}u_{331}-\sum_{i\neq 1} \frac{f^3}{\lambda_1\sigma_1^4} \frac{u_{111}^2}{(F^{ii})^2}+2\sum_{i\neq 1}\frac{F^{ii}u_{ii1}^2}{\lambda_1 (\lambda_1 -\lambda_i )}\\\nonumber
&\ +\sum_{i\neq 1}\frac{F^{ii}u_{11i}^2}{\lambda_1^2}+\frac{F^{11}u_{111}^2}{\lambda_1\sigma_1}-C.
\end{align}

Let us compute the coefficient of $u_{111}^2$. Using the fact that $F^{33}\geq F^{22}$, $F^{ii}=f\left(\frac{1}{\lambda_i}-\frac{1}{\sigma_1}\right)$ and $\lambda_2\leq \frac{\lambda_1}{4}$, we have
\begin{align*}
&\ -\sum_{i\neq 1} \frac{f^3}{\lambda_1\sigma_1^4(F^{ii})^2 } +\frac{F^{11}}{\lambda_1\sigma_1}\geq -\frac{2f^3}{\lambda_1\sigma_1^4 (F^{22})^2 } +\frac{F^{11}}{\lambda_1\sigma_1}\\
=&\ \frac{f^3}{\lambda_1\sigma_1^4(F^{22})^2}\left(\sigma_1^3\left(\frac{1}{\lambda_1}-\frac{1}{\sigma_1}\right)\left(\frac{1}{\lambda_2}-\frac{1}{\sigma_1}\right)^2-2\right)\\
=&\ \frac{f^3}{\lambda_1\sigma_1^4(F^{22})^2} \frac{(\lambda_2+\lambda_3)(\lambda_1+\lambda_3)^2-2\lambda_1\lambda_2^2}{\lambda_1\lambda_2^2} \\
\geq &\ \frac{f^3}{\lambda_1\sigma_1^4(F^{22})^2} \frac{\lambda_1^2\lambda_2-2\lambda_1\lambda_2^2}{\lambda_1\lambda_2^2} \\
\geq &\ \frac{f^3}{\lambda_1\sigma_1^4(F^{22})^2}  \frac{\lambda_1}{2\lambda_2}=\frac{f^3}{2\lambda_2\sigma_1^4(F^{22})^2}.
\end{align*}

Using the fact that $F^{22}=f\left(\frac{1}{\lambda_2}-\frac{1}{\sigma_1}\right)\leq \frac{f}{\lambda_2}$ and $F^{11}=f\left(\frac{1}{\lambda_1}-\frac{1}{\sigma_1}\right)=\frac{f(\lambda_2+\lambda_3)}{\lambda_1\sigma_1}$, we have
\begin{align*}
\frac{f^3}{2\lambda_2\sigma_1^4(F^{22})^2} \geq \frac{f\lambda_2}{2\sigma_1^4}\geq  \frac{f(\lambda_2+\lambda_3)}{4\sigma_1^4}= \frac{\lambda_1F^{11}}{4\sigma_1^3} \geq \frac{1}{108}\frac{F^{11}}{\lambda_1^2}.
\end{align*}

Combining the above two inequalities and plugging into (\ref{m-6}), we have
\begin{align}\label{m-7}
\sum_i F^{ii}b_{ii}\geq&\ -\frac{2f}{\lambda_1\sigma_1^2} u_{221}u_{331}+2\sum_{i\neq 1}\frac{F^{ii}u_{ii1}^2}{\lambda_1 (\lambda_1 -\lambda_i )}+\sum_{i\neq 1}\frac{F^{ii}u_{11i}^2}{\lambda_1^2}\\\nonumber
&\ +\frac{1}{108}\frac{F^{11}u_{111}^2}{\lambda_1^2}-C.
\end{align}

Note that 
\begin{align*}
&\ -\frac{2f}{\lambda_1\sigma_1^2} u_{221}u_{331}+2\sum_{i\neq 1}\frac{F^{ii}u_{ii1}^2}{\lambda_1 (\lambda_1 -\lambda_i )}\\
\geq &\ -\frac{2f}{\lambda_1\sigma_1^2} u_{221}u_{331}+\frac{2F^{22}u_{221}^2}{\lambda_1^2}+\frac{2F^{33}u_{331}^2}{\lambda_1^2}\\
=&\ \left( \frac{\sqrt{2F^{22} }u_{221} }{\lambda_1}-\frac{f}{\sqrt{2} \sigma_1^2}\frac{u_{331}}{\sqrt{F^{22}} }\right)^2-\frac{f^2}{2\sigma_1^4}\frac{u_{331}^2}{F^{22}}+\frac{2F^{33}u_{331}^2}{\lambda_1^2}\\
\geq &\ \left(\frac{2F^{33}}{\lambda_1^2}-\frac{f^2}{2\sigma_1^4F^{22}}\right) u_{331}^2.
\end{align*}

Plugging into (\ref{m-7}), we have
\begin{align}\label{m-8}
\sum_i F^{ii}b_{ii}\geq&\ \left(\frac{2F^{33}}{\lambda_1^2}-\frac{f^2}{2\sigma_1^4F^{22}}\right) u_{331}^2+\sum_{i\neq 1}\frac{F^{ii}u_{11i}^2}{\lambda_1^2}+\frac{1}{108}\frac{F^{11}u_{111}^2}{\lambda_1^2}-C.
\end{align}

Since $\sigma_1\geq \lambda_2+\lambda_3\geq  2\lambda_3$, we have
\begin{align}\label{m-8-1}
F^{33}=f\left(\frac{1}{\lambda_3}-\frac{1}{\sigma_1}\right)\geq \frac{f}{2\lambda_3}.
\end{align}

Together with the fact that $F^{22}=\frac{f(\lambda_1+\lambda_3)}{\lambda_2\sigma_1}$, $\lambda_3=\frac{f\sigma_1}{\lambda_!\lambda_2}$ and Lemma \ref{lambda_3}, we have
\begin{align*}
\frac{2F^{33}}{\lambda_1^2}-\frac{f^2}{2\sigma_1^4F^{22}}\geq&\ \frac{f}{\lambda_1^2\lambda_3}-\frac{f\lambda_2}{2\sigma_1^3(\lambda_1+\lambda_3)}\geq \frac{\lambda_2}{\lambda_1\sigma_1}-\frac{f\lambda_2}{2\lambda_1\sigma_1^3}=\frac{\lambda_2}{\lambda_1\sigma_1} \left(1-\frac{f}{2\sigma_1^2} \right)\geq 0.
\end{align*}
We have used the fact that $\sigma_1\geq \lambda_1+\lambda_2\geq 2\lambda_2\geq 2\sqrt{\frac{f}{3}}\geq \sqrt{\frac{f}{2}}$ in the last inequality.

Plugging into (\ref{m-8}), we conclude that 
\begin{align*}
\sum_i F^{ii}b_{ii}\geq \sum_{i\neq 1}\frac{F^{ii}u_{11i}^2}{\lambda_1^2}+\frac{1}{108}\frac{F^{11}u_{111}^2}{\lambda_1^2}-C\geq  \frac{1}{108} \sum_i F^{ii}b_i^2-C.
\end{align*}

This completes the proof for Case A.

\medskip

{\bf Case B:} $\lambda_2\geq \frac{\lambda_1}{4}$.

\medskip

In this case, we have
\begin{align*}
\frac{1}{\lambda_1-\lambda_2}\geq \frac{1}{3\lambda_2}\geq \frac{1}{3}\left(\frac{1}{\lambda_2}-\frac{1}{\sigma_1}\right)=\frac{1}{3}\frac{F^{22}}{f}.
\end{align*}

Plugging into (\ref{m-5}), we have
\begin{align}\label{m-9}
\sum_i F^{ii}b_{ii}\geq&\ -\frac{f}{\lambda_1\sigma_1^2} \sum_{i\neq j}u_{ii1}u_{jj1}+ \frac{4(F^{22})^2u_{221}^2}{3f\lambda_1}+\frac{F^{22}u_{221}^2}{\lambda_1(\lambda_1-\lambda_2)}+ \frac{(F^{33})^2u_{331}^2}{f\lambda_1}\\\nonumber
&\ +\frac{2F^{33}u_{331}^2}{\lambda_1(\lambda_1-\lambda_3)}+\sum_{i\neq 1}\frac{F^{ii}u_{11i}^2}{\lambda_1^2} +\frac{F^{11}u_{111}^2}{\lambda_1\sigma_1}-C.
\end{align}

Note that
\begin{align*}
&\ -\frac{2f}{\lambda_1\sigma_1^2}u_{111}u_{221}+\frac{4(F^{22})^2u_{221}^2}{3f\lambda_1}\\
=&\ \left(\frac{2F^{22}u_{221}}{\sqrt{3f\lambda_1}}-\frac{f\sqrt{3f}}{2\sqrt{\lambda_1}\sigma_1^2}\frac{u_{111}}{F^{22}}\right)^2- \frac{3f^3}{4\lambda_1\sigma_1^4} \frac{u_{111}^2}{(F^{22})^2}\\
\geq&\ - \frac{3f^3}{4\lambda_1\sigma_1^4} \frac{u_{111}^2}{(F^{22})^2},
\end{align*}

Together with (\ref{m-5-1}), we have
\begin{align*}
&\ -\frac{f}{\lambda_1\sigma_1^2} \sum_{i\neq j}u_{ii1}u_{jj1}+\frac{4(F^{22})^2u_{221}^2}{3f\lambda_1}+ \frac{(F^{33})^2u_{331}^2}{f\lambda_1}\\
\geq &\ -\frac{2f}{\lambda_1\sigma_1^2} u_{221}u_{331}- \frac{3f^3}{4\lambda_1\sigma_1^4} \frac{u_{111}^2}{(F^{22})^2}- \frac{f^3}{\lambda_1\sigma_1^4} \frac{u_{111}^2}{(F^{33})^2}.
\end{align*}

Plugging into (\ref{m-9}), we have
\begin{align}\label{m-10}
\sum_i F^{ii}b_{ii}\geq&\ -\frac{2f}{\lambda_1\sigma_1^2} u_{221}u_{331}- \frac{3f^3}{4\lambda_1\sigma_1^4} \frac{u_{111}^2}{(F^{22})^2}- \frac{f^3}{\lambda_1\sigma_1^4} \frac{u_{111}^2}{(F^{33})^2}+\frac{F^{22}u_{221}^2}{\lambda_1(\lambda_1-\lambda_2)}\\\nonumber
&\ +\frac{2F^{33}u_{331}^2}{\lambda_1(\lambda_1-\lambda_3)}+\sum_{i\neq 1}\frac{F^{ii}u_{11i}^2}{\lambda_1^2}+\frac{F^{11}u_{111}^2}{\lambda_1\sigma_1}-C.
\end{align}

Let us compute the coefficient of $u_{111}^2$. Using the fact that $F^{ii}=f\left(\frac{1}{\lambda_i}-\frac{1}{\sigma_1}\right)$, we have
\begin{align*}
&\ -\frac{3f^3}{4\lambda_1\sigma_1^4(F^{22})^2} -\frac{f^3}{\lambda_1\sigma_1^4(F^{33})^2} +\frac{F^{11}}{\lambda_1\sigma_1}\\
=&\ \frac{f^3}{\lambda_1\sigma_1^4(F^{22})^2}\left(\sigma_1^3\left(\frac{1}{\lambda_1}-\frac{1}{\sigma_1}\right)\left(\frac{1}{\lambda_2}-\frac{1}{\sigma_1}\right)^2-\frac{3}{4}\right) -\frac{f^3}{\lambda_1\sigma_1^4} \frac{1}{(F^{33})^2}\\
=&\ \frac{f^3}{\lambda_1\sigma_1^4(F^{22})^2} \frac{4(\lambda_2+\lambda_3)(\lambda_1+\lambda_3)^2-3\lambda_1\lambda_2^2}{4\lambda_1\lambda_2^2} -\frac{f^3}{\lambda_1\sigma_1^4} \frac{1}{(F^{33})^2}\\
\geq &\ \frac{f^3}{\lambda_1\sigma_1^4(F^{22})^2} \frac{4\lambda_1^2\lambda_2-3\lambda_1\lambda_2^2}{4\lambda_1\lambda_2^2} -\frac{f^3}{\lambda_1\sigma_1^4} \frac{1}{(F^{33})^2}\\
\geq &\ \frac{f^3}{\lambda_1\sigma_1^4(F^{22})^2} \frac{\lambda_1}{4\lambda_2}-\frac{f^3}{\lambda_1\sigma_1^4} \frac{1}{(F^{33})^2}.
\end{align*}

By (\ref{m-8-1}) and the fact that $F^{22}=f\left(\frac{1}{\lambda_2}-\frac{1}{\sigma_1}\right)\leq \frac{f}{\lambda_2}$ and $\lambda_3=\frac{f\sigma_1}{\lambda_1\lambda_2}$, we have
\begin{align*}
&\ \frac{f^3}{\lambda_1\sigma_1^4(F^{22})^2} \frac{\lambda_1}{4\lambda_2}-\frac{f^3}{\lambda_1\sigma_1^4} \frac{1}{(F^{33})^2}\\
=&\ \frac{f^3}{\lambda_1\sigma_1^4} \left(\frac{\lambda_1}{4\lambda_2(F^{22})^2}-\frac{1}{(F^{33})^2}\right)\\
\geq &\ \frac{f^3}{\lambda_1\sigma_1^4} \left(\frac{\lambda_1\lambda_2}{4f^2}-\frac{4\lambda_3^2}{f^2}\right)\\
=&\ \frac{f}{\lambda_1\sigma_1^4} \left(\frac{\lambda_1\lambda_2}{4}-\frac{4f^2\sigma_1^2}{\lambda_1^2\lambda_2^2}\right)\\
\geq &\ \frac{f}{\lambda_1\sigma_1^4} \left(\frac{\lambda_1\lambda_2}{4}-\frac{36f^2}{\lambda_2^2}\right)\\
\geq &\ \frac{f}{\lambda_1\sigma_1^4} \frac{\lambda_1\lambda_2}{8},
\end{align*}
by assuming $\lambda_1\lambda_2^3\geq 288 f^2$. Since $\lambda_2\geq \sqrt{\frac{f}{3}}$, it suffice to assume $\lambda_1\geq\Lambda(f)\geq  864\sqrt{3f}$.

Combining with the fact that $F^{11}=\frac{f(\lambda_2+\lambda_3)}{\lambda_1\sigma_1}$, we have
\begin{align*}
\frac{f}{\lambda_1\sigma_1^4} \frac{\lambda_1\lambda_2}{8}= &\ \frac{f\lambda_2}{8\sigma_1^4}\geq \frac{f(\lambda_2+\lambda_3)}{16\sigma_1^4}=\frac{\lambda_1 F^{11}}{16\sigma_1^3}\geq \frac{1}{432}\frac{F^{11}}{\lambda_1^2}.
\end{align*}

Combining the above three inequalities and plugging into (\ref{m-10}), we have
\begin{align}\label{m-11}
\sum_i F^{ii}b_{ii}\geq&\ -\frac{2f}{\lambda_1\sigma_1^2} u_{221}u_{331}+\frac{F^{22}u_{221}^2}{\lambda_1(\lambda_1-\lambda_2)}+\frac{2F^{33}u_{331}^2}{\lambda_1(\lambda_1-\lambda_3)}\\\nonumber
&\ +\sum_{i\neq 1}\frac{F^{ii}u_{11i}^2}{\lambda_1^2}+\frac{1}{432}\frac{F^{11}u_{111}^2}{\lambda_1^2}-C.
\end{align}

Note that 
\begin{align*}
&\ -\frac{2f}{\lambda_1\sigma_1^2} u_{221}u_{331}+\frac{F^{22}u_{221}^2}{\lambda_1(\lambda_1-\lambda_2)}+\frac{2F^{33}u_{331}^2}{\lambda_1(\lambda_1-\lambda_3)}\\
\geq &\ -\frac{2f}{\lambda_1\sigma_1^2} u_{221}u_{331}+\frac{F^{22}u_{221}^2}{\lambda_1^2}+\frac{2F^{33}u_{331}^2}{\lambda_1^2}\\
=&\ \left(\frac{\sqrt{F^{22}}u_{221}}{\lambda_1}-\frac{f}{\sigma_1^2}\frac{u_{331}}{\sqrt{F^{22}}}\right)^2 -\frac{f^2}{\sigma_1^4 }\frac{u_{331}^2}{F^{22}}+\frac{2F^{33}u_{331}^2}{\lambda_1^2}\\
\geq&\  \left( \frac{2F^{33}}{\lambda_1^2} -\frac{f^2}{\sigma_1^4 F^{22}}\right) u_{331}^2.
\end{align*}

Plugging into (\ref{m-11}), we have 
\begin{align}\label{m-12}
\sum_i F^{ii}b_{ii}\geq&\  \left( \frac{2F^{33}}{\lambda_1^2} -\frac{f^2}{\sigma_1^4 F^{22}}\right) u_{331}^2+\sum_{i\neq 1}\frac{F^{ii}u_{11i}^2}{\lambda_1^2}+\frac{1}{432}\frac{F^{11}u_{111}^2}{\lambda_1^2}-C.
\end{align}

By (\ref{m-8-1}) and the fact that $F^{22}=\frac{f(\lambda_1+\lambda_3)}{\lambda_2\sigma_1}$, $\lambda_3=\frac{f\sigma_1}{\lambda_1\lambda_2}$ and Lemma \ref{lambda_3}, we have
\begin{align*}
\frac{2F^{33}}{\lambda_1^2} -\frac{f^2}{\sigma_1^4 F^{22}}\geq \frac{f}{\lambda_1^2\lambda_3}-\frac{f\lambda_2}{\sigma_1^3(\lambda_1+\lambda_3)}\geq \frac{\lambda_2}{\lambda_1\sigma_1}-\frac{f\lambda_2}{\lambda_1\sigma_1^3}=\frac{\lambda_2}{\lambda_1\sigma_1}\left(1-\frac{f}{\sigma_1^2}\right)\geq 0.
\end{align*}
We have used the fact that $\sigma_1\geq \lambda_1+\lambda_2\geq 2\lambda_2\geq 2\sqrt{\frac{f}{3}}\geq \sqrt{f}$ in the last inequality.

Plugging into (\ref{m-12}), we conclude that 
\begin{align*}
\sum_i F^{ii}b_{ii}\geq&\ \sum_{i\neq 1}\frac{F^{ii}u_{11i}^2}{\lambda_1^2}+\frac{1}{432}\frac{F^{11}u_{111}^2}{\lambda_1^2}-C\geq \frac{1}{432}\sum_i F^{ii}b_i^2-C.
\end{align*}

This completes the proof for Case B. The proof for the case that $\lambda_1$ has multiplicity $1$ is now finished.

\end{proof}

\subsection{$\lambda_1$ has multiplicity $2$}

In this subsection, we will prove Lemma \ref{Jacobi} for the case that $\lambda_1$ has multiplicity $2$, i.e. $\lambda_1=\lambda_2>\lambda_3$. The proof is very similar to the case that $\lambda_1$ has multiplicity $1$. For the sake of completeness, we provide all details here.

\begin{proof}

By Lemma 5 in \cite{BCD}, at $x_0$, we have
\begin{align}\label{BCD}
\delta_{kl}\cdot{\lambda _1}_i=u_{kli},\quad 1\leq k,l\leq 2,
\end{align}
\begin{align*}
{\lambda_1}_{ii}\geq u_{11ii}+2\frac{u_{13i}^2}{\lambda_1 -\lambda_3},
\end{align*}
in the viscosity sense.

It follows that
\begin{align*}
b_{ii}=\frac{{\lambda_1}_{ii}}{\lambda_1}-\frac{{\lambda_1}_i^2}{\lambda_1^2}= \frac{u_{11ii}}{ \lambda_1}+2\frac{u_{13i}^2}{\lambda_1 (\lambda_1 -\lambda_3 )}-\frac{ u_{11i}^2}{\lambda_{1}^2}.
\end{align*}

Contracting with $F^{ii}$, we have
\begin{align}\label{m-1-1}
\sum_i F^{ii}b_{ii}\geq \sum_i \frac{F^{ii}u_{11ii}}{ \lambda_1}+\sum_i\frac{2F^{ii}u_{13i}^2}{\lambda_1 (\lambda_1 -\lambda_3 )}-\sum_i \frac{F^{ii} u_{11i}^2}{\lambda_{1}^2}.
\end{align}

Plugging (\ref{m-0-1}) into (\ref{m-1-1}), we have
\begin{align}\label{m-1-3}
\sum_i F^{ii}b_{ii}\geq &\  -\sum_{p,q,r,s} \frac{F^{pq,rs}u_{pq1}u_{rs1}}{\lambda_1 }+\sum_i\frac{2F^{ii}u_{13i}^2}{\lambda_1 (\lambda_1 -\lambda_3 )}-\sum_i \frac{ F^{ii}u_{11i}^2}{\lambda_{1}^2}-C,
\end{align}
where $C$ depends only on $\|u\|_{C^{0,1}(\overline{B}_9)}$ and $\|f\|_{C^{1,1}\left( \overline{B}_9\times [-M,M] \right) }$. Here $M$ is a large constant satisfying $\|u\|_{L^\infty(\overline{B}_9)}\leq M$. 

In the following, we will denote $C$ to be a constant depending only on $\|u\|_{C^{0,1}(\overline{B}_9)}$, $\min_{\overline{B}_9\times [-M,M] }f $ and $\|f\|_{C^{1,1}\left( \overline{B}_9\times [-M,M]\right) }$. It may change from line to line. 

Plugging (\ref{m-0-2}) into (\ref{m-1-3}), we have
\begin{align}\label{m-1-4}
\sum_i F^{ii}b_{ii}\geq &\ -\sum_{i, j} \frac{F^{ii,jj}u_{ii1}u_{jj1}}{\lambda_1} -\sum_{i\neq j} \frac{F^{ij,ji}u_{ij1}^2}{\lambda_1}+\sum_i\frac{2F^{ii}u_{13i}^2}{\lambda_1 (\lambda_1 -\lambda_3 )}\\\nonumber
&\ -\sum_i \frac{ F^{ii}u_{11i}^2}{\lambda_{1}^2}-C.
\end{align}

Since $\lambda_1=\lambda_2>\lambda_3$, by Lemma \ref{F^{ijkl}}, we have
\begin{align*}
-\sum_{i\neq j} \frac{F^{ij,ji}u_{ij1}^2}{\lambda_1}=\sum_{i\neq j}\frac{\sigma_3^{ii,jj}u_{ij1}^2}{\lambda_1\sigma_1} \geq 2\frac{\sigma_3^{11,33}u_{113}^2}{\lambda_1\sigma_1} = 2\frac{(F^{33}-F^{11})u_{113}^2}{\lambda_1(\lambda_1-\lambda_3)}.
\end{align*}

On the other hand,
\begin{align*}
\sum_i\frac{2F^{ii}u_{13i}^2}{\lambda_1 (\lambda_1 -\lambda_3 )}\geq \frac{2F^{11}u_{131}^2}{\lambda_1 (\lambda_1 -\lambda_3 )}.
\end{align*}

Combining the above two inequalities and plugging into (\ref{m-1-4}), we have
\begin{align*}
\sum_i F^{ii}b_{ii}\geq &\ -\sum_{i,j} \frac{F^{ii,jj} u_{ii1}u_{jj1}}{\lambda_1} +\frac{2F^{33}u_{113}^2}{\lambda_1 (\lambda_1 -\lambda_3 )}-\sum_i \frac{ F^{ii}u_{11i}^2}{\lambda_{1}^2}-C.
\end{align*}

By (\ref{BCD}), $u_{221}=u_{122}=\delta_{12}(\lambda_1)_2=0$, $u_{112}=u_{121}=\delta_{12}(\lambda_1)_1=0$. Consequently,
\begin{align*}
\sum_i F^{ii}b_{ii}\geq &\ -\frac{F^{11,11}u_{111}^2}{\lambda_1}- \frac{F^{33,33}u_{331}^2}{\lambda_1}- 2\frac{F^{11,33}u_{111}u_{331}}{\lambda_1} \\\nonumber
&\ +\frac{ F^{33}u_{113}^2}{\lambda_1^2} -\frac{ F^{11}u_{111}^2}{\lambda_{1}^2}-C.
\end{align*}

By (\ref{m-3-1}), we have $F^{ii,ii}=-\frac{2}{\sigma_1}F^{ii}$. Consequently,
\begin{align}\label{m-1-5}
\sum_i F^{ii}b_{ii}\geq&\ -2\frac{F^{11,33}u_{111}u_{331}}{\lambda_1}+\frac{ F^{33}u_{113}^2}{\lambda_1^2} -\frac{1}{\lambda_1}\left(\frac{1}{\lambda_1}-\frac{2}{\sigma_1}\right)F^{11}u_{111}^2-C.
\end{align}

By (\ref{m-4-1}), we have
\begin{align*}
F^{11,33}=\frac{F^{11}F^{33}}{f}+\frac{f}{\sigma_1^2}
\end{align*}

Together with the fact that $\sum_i F^{ii}u_{ii1}=F^{11}u_{111}+F^{33}u_{331}=f_1$ and $F^{11}=f\left(\frac{1}{\lambda_1}-\frac{1}{\sigma_1}\right)$, we have
\begin{align*}
&\ -2 \frac{F^{11,33}u_{111}u_{331}}{\lambda_1} \\
=&\ -\frac{2}{f\lambda_1} F^{11}F^{33}u_{111}u_{331}-\frac{2f}{\lambda_1\sigma_1^2}u_{111}u_{331}\\
=&\ -\frac{ \left( F^{11}u_{111}+F^{33}u_{331}\right)^2}{f\lambda_1}+\frac{ (F^{11})^2u_{111}^2}{f\lambda_1}+\frac{(F^{33})^2u_{331}^2}{f\lambda_1}-\frac{2f}{\lambda_1\sigma_1^2}u_{111}u_{331}\\
=&\ -\frac{f_1^2}{f\lambda_1}+\frac{(F^{11})^2u_{111}^2}{f\lambda_1}+\frac{(F^{33})^2u_{331}^2}{f\lambda_1}-\frac{2f}{\lambda_1\sigma_1^2}u_{111}u_{331}\\
\geq &\ \frac{(F^{33})^2u_{331}^2}{f\lambda_1}+\frac{1}{\lambda_1}\left(\frac{1}{\lambda_1}-\frac{1}{\sigma_1}\right)F^{11}u_{111}^2-\frac{2f}{\lambda_1\sigma_1^2}u_{111}u_{331}-C.
\end{align*}

Plugging into (\ref{m-1-5}), we have
\begin{align*}
\sum_i F^{ii}b_{ii}\geq&\ -\frac{2f}{\lambda_1\sigma_1^2}u_{111}u_{331}+\frac{(F^{33})^2u_{331}^2}{f\lambda_1}+\frac{ F^{33}u_{113}^2}{\lambda_1^2}+\frac{F^{11}u_{111}^2}{\lambda_1\sigma_1}-C.
\end{align*}

Together with (\ref{m-5-1}), we have
\begin{align}\label{m-1-6}
\sum_i F^{ii}b_{ii}\geq&\ \left(\frac{F^{11}}{\lambda_1\sigma_1}-\frac{f^3}{\lambda_1\sigma_1^4(F^{33})^2}\right)u_{111}^2+\frac{ F^{33}u_{113}^2}{\lambda_1^2} -C.
\end{align}

By (\ref{m-8-1}) and the fact that $\lambda_3=\frac{f\sigma_1}{\lambda_1\lambda_2}=\frac{f\sigma_1}{\lambda_1^2}$, $F^{11}=\frac{f(\lambda_2+\lambda_3)}{\lambda_1\sigma_1}=\frac{f(\lambda_1+\lambda_3)}{\lambda_1\sigma_1}$,  we have
\begin{align*}
\frac{f^3}{\lambda_1\sigma_1^4(F^{33})^2} \leq \frac{4f\lambda_3^2}{\lambda_1\sigma_1^4}=\frac{4f^3}{\lambda_1^5\sigma_1^2}\leq \frac{f}{2\lambda_1\sigma_1^2} \leq \frac{f(\lambda_1+\lambda_3)}{2\lambda_1^2\sigma_1^2} =\frac{F^{11}}{2\lambda_1\sigma_1},
\end{align*}
by assuming $\lambda_1^4\geq 8f^2$. It suffice to assume $\lambda_1\geq \Lambda(f)\geq \sqrt[4]{8f^2}$.

Plugging into (\ref{m-1-6}), we conclude that 
\begin{align*}
\sum_i F^{ii}b_{ii}\geq&\ \frac{F^{11}u_{111}^2}{2\lambda_1\sigma_1}+\frac{ F^{33}u_{113}^2}{\lambda_1^2}-C\geq \frac{1}{6} \sum_i F^{ii}b_i^2-C.
\end{align*}

The proof for the case that $\lambda_1$ has multiplicity $2$ is now finished.
\end{proof}

\section{Legendre transform}

In this section, we adopt the idea by Shankar and Yuan \cite{SY-20} to use Legendre transform to obtain the bound of $b$ via its integral. 

For each $K>0$, define the Legendre transform for the function $u+\frac{K}{2}|x|^2$. We have
\begin{align*}
(x,Du(x)+Kx)=(Dw(y),y),
\end{align*}
where $w(y)$ is the Legendre transform for the function $u+\frac{K}{2}|x|^2$. Note that $y(x)=Du(x)+Kx$ is a diffeomorphism.

Define 
\begin{align*}
G(D^2w(y))=-F\left(-KI+\left(D^2w(y)\right)^{-1}\right)=-F(D^2u(x)).
\end{align*}

Suppose $D^2u$ is diagonalized at $x_0$. Then at $x_0$, we have
\begin{align}\label{g}
G^{ii}=F^{ii}\cdot w_{ii}^{-2}=F^{ii}(K+u_{ii})^2.
\end{align} 

\begin{lemm}\label{Jacobi-g}
Let $f\in C^{1,1}(B_{10}\times \mathbb{R})$ be a positive function and let $u\in C^4(B_{10})$ be a convex solution of (\ref{eq}). For any $x_0\in B_9$, suppose that $D^2u$ is diagonalized at $x_0$ such that $\lambda_1\geq \lambda_2\geq \lambda_3$ and $\lambda_1\geq \Lambda(f)$, where $\Lambda(f)$ is a large constant depending only on $\max_{\overline{B}_9\times [-M,M]}f$. Set $b=\ln\lambda_1$. Then at $y(x_0)$, we have
\begin{align*}
\sum_i G^{ii}b^*_{ii}\geq -C,
\end{align*}
in the viscosity sense, where $C$ depends only on $K$, $\|u\|_{C^{0,1}(\overline{B}_9)}$, $\min_{\overline{B}_9\times [-M,M]}f $ and $\|f\|_{C^{1,1}\left( \overline{B}_9\times [-M,M]\right) }$. Here $M$ is a large constant satisfying $\|u\|_{L^{\infty}(\overline{B}_9)}\leq M$. We use the notation $b^*(y)=b(y(x))$.
\end{lemm}

\begin{proof}
By chain rule, we have
\begin{align}\label{chain}
\frac{\partial b}{\partial x_i}=\sum_k\frac{\partial y^k}{\partial x^i}\frac{\partial b^*}{\partial y^k}=\sum_k(K\delta_{ik}+u_{ik})b^*_k.
\end{align}

Consequently,
\begin{align*}
\frac{\partial^2 b}{\partial x_i^2}=&\ \sum_k u_{iik}b^*_k+\sum_{k,l}(K\delta_{ik}+u_{ik})b^*_{kl}(K\delta_{il}+u_{il})\\
=&\ \sum_k u_{iik}b^*_k+(K+u_{ii})^2b^*_{ii}.
\end{align*}

Together with (\ref{g}), we have
\begin{align*}
\sum_iF^{ii}b_{ii}=&\ \sum_{i,k}F^{ii}u_{iik}b^*_k+\sum_{i}F^{ii}(K+u_{ii})^2b^*_{ii}\\
=&\ \sum_kf_kb^*_k+\sum_iG^{ii}b^*_{ii}.
\end{align*}

Together with Lemma \ref{Jacobi}, we have
\begin{align*}
\sum_iG^{ii}b^*_{ii}=\sum_iF^{ii}b_{ii}-\sum_kf_kb^*_k\geq \frac{1}{432}\sum_iF^{ii}b_i^2-\sum_kf_kb^*_k-C.
\end{align*}

By (\ref{g}) and (\ref{chain}), we have
\begin{align*}
\sum_iF^{ii}b_i^2=\sum_i G^{ii}(K+u_{ii})^{-2}\bigg( (K+u_{ii})b_i^*\bigg)^2=\sum_iG^{ii}(b^*_i)^2.
\end{align*}

It follows that
\begin{align*}
\sum_iG^{ii}b^*_{ii}\geq&\ \frac{1}{432}\sum_iG^{ii}(b^*_i)^2-\sum_kf_kb^*_k-C\\
\geq&\  \frac{1}{432}\sum_iG^{ii}(b^*_i)^2-C\sum_k|b^*_k|-C\\
\geq&\ -C\sum_i \frac{1}{G^{ii}}-C.
\end{align*}

By (\ref{g}), Lemma \ref{F^{ijkl}}, Lemma \ref{lambda_3} and the fact that $\lambda_3=\frac{f\sigma_1}{\lambda_1\lambda_2}$, we have
\begin{align}\label{g-2}
G^{11}=&\ \frac{f(\lambda_2+\lambda_3)}{\lambda_1\sigma_1}(K+\lambda_1)^2\in \left(C^{-1}\lambda_2,C\lambda_2\right),\\\nonumber
G^{22}=&\ \frac{f(\lambda_1+\lambda_3)}{\lambda_2\sigma_1}(K+\lambda_2)^2\in \left(C^{-1}\lambda_2,C\lambda_2\right),\\\nonumber
G^{33}=&\ \frac{f(\lambda_1+\lambda_2)}{\lambda_3\sigma_1}(K+\lambda_3)^2= \frac{\lambda_1\lambda_2(\lambda_1+\lambda_2)}{\sigma_1^2}(K+\lambda_3)^2\in \left(C^{-1}\lambda_2,C\lambda_2\right).
\end{align}

Together with Lemma \ref{lambda_3}, we have
\begin{align*}
-\sum_i\frac{1}{G^{ii}}\geq -\frac{C}{\lambda_2}\geq -C.
\end{align*}

It follows that 
\begin{align*}
\sum_iG^{ii}b^*_{ii}\geq -C.
\end{align*}

The lemma is now proved.
\end{proof}

We now state the mean value inequality. 
\begin{lemm}\label{mean}
Let $f\in C^{1,1}(B_{10}\times \mathbb{R})$ be a positive function and let $u\in C^4(B_{10})$ be a convex solution of (\ref{eq}). Let
\begin{align*}
b=\ln \max\{\lambda_1,\Lambda(f)\},
\end{align*}
where $\lambda_1$ is the largest eigenvalue of $D^2u$ and $\Lambda(f)$ is from Lemma \ref{Jacobi-g}. Then we have
\begin{align*}
b(0)\leq  C\int_{B_1}b(x)\sigma_ 2(D^2u(x)) dx+C,
\end{align*}
where $C$ depends only on $\|u\|_{C^{0,1}(\overline{B}_9)}$, $\min_{\overline{B}_9\times [-M,M]}f $ and $\|f\|_{C^{1,1}\left( \overline{B}_9\times [-M,M]\right) }$. Here $M$ is a large constant satisfying $\|u\|_{L^\infty(\overline{B}_9)}\leq M$. 
\end{lemm}

\begin{proof}
Let $K=1$ in Lemma \ref{Jacobi-g}. Without loss of generality, we may assume $Du(0)=0$, i.e. $y(0)=0$.  Note that the maximum of two subsolutions is still a subsolution. It follows that $b$ satisfies
\begin{align*}
\sum_i G^{ii}b^*_{ii}\geq -C,
\end{align*}
in the viscosity sense.

By Lemma \ref{lambda_3} and (\ref{g-2}), we have
\begin{align*}
\sum_i G^{ii}\geq G^{11}=\frac{f(\lambda_2+\lambda_3)}{\lambda_1\sigma_1}(1+\lambda_1)^2\geq \frac{f\lambda_1\lambda_2}{\sigma_1}\geq \frac{f}{3}\sqrt{\frac{f}{3}}.
\end{align*}

Consequently,
\begin{align*}
\sum_i G^{ii}\left( b^*(y)+A|y|^2\right)_{ii}\geq 0,
\end{align*}
where $A$ depends only on $\|u\|_{C^{0,1}(\overline{B}_9)}$, $\min_{\overline{B}_9\times [-M,M]}f $ and $\|f\|_{C^{1,1}\left( \overline{B}_9\times [-M,M]\right) }$.

By (\ref{g-2}), $\frac{G^{ii}}{\lambda_2} $ is uniformly elliptic. By local maximum principle (Theorem 4.8 in \cite{CC}), we have
\begin{align*}
b^*(0)+A|0|^2\leq C\int_{B_1^y}\left( b^*(y)+A|y|^2\right) dy.
\end{align*}

Since $D^2u>0$, $y(x)=Du(x)+x$ is uniformly monotone, i.e. $|y(x_I)-y(x_{II})|\geq |x_I-x_{II}|$, we have $x(B_1^y)\subset B_1$. Together with the fact that $y(0)=0$ and $\sigma_3=f\sigma_1$, we have
\begin{align*}
b(0)= b^*(0)\leq &\ C\int_{B_1}b(x)\det(D^2u(x)+I)dx+C\\
\leq &\ C\int_{B_1}b(x) (\lambda_1+1)(\lambda_2+1)(\lambda_3+1)) dx+C\\
\leq &\ C\int_{B_1}b\left(\sigma_1+\sigma_2+\sigma_3\right)  dx+C\\
\leq &\ C\int_{B_1}b\left(\sigma_1+\sigma_2\right)  dx+C.
\end{align*}

By Lemma \ref{lambda_3}, we have
\begin{align*}
\sigma_2\geq \lambda_1\lambda_2\geq \frac{\sigma_1}{3}\sqrt{\frac{f}{3}}.
\end{align*}

It follows that
\begin{align*}
b(0)\leq C\int_{B_1}b\sigma_2 dx+C.
\end{align*}

The lemma is now proved.
\end{proof}

\section{Proof of the main theorem}

In this section, we complete the proof of Theorem \ref{Theorem} via integration by parts. Unlike Hessian equations, the coefficient $F^{ij}$ of the linearized equation is not divergence free. We will choose $H^{ij}=\sigma_1F^{ij}$ to perform the integration by parts.

\begin{proof}
Let $b=\ln \max\{\lambda_1,\Lambda(f)\}$ as in Lemma \ref{mean}.

Let $\varphi$ be a cutoff function such that $\varphi=1$ on $B_1$ and $\varphi=0$ outside $B_2$. Using the fact that $\sum_i (\sigma_2^{ij})_i=0$, we have
\begin{align*}
\int_{B_1}b\sigma_ 2dx\leq&\ \int_{B_2}\varphi b \sigma_ 2dx=\frac{1}{2}\sum_{i,j}\int_{B_2}\varphi b \sigma_ 2^{ij}u_{ij}dx\\\nonumber
=&\ -\frac{1}{2}\sum_{i,j}\int_{B_2} (\varphi_i b+\varphi b_i) \sigma_2^{ij}u_j dx\\\nonumber
\leq &\ C\int_{B_2} b\sigma_1 dx-\frac{1}{2}\sum_{i,j}\int_{B_2}\varphi b_i\sigma_2^{ij}u_j dx,
\end{align*}
where $C$ depends only on $\|u\|_{C^{0,1}(\overline{B}_9)}$.

In the following, we will denote $C$ to be a constant depending only on $\|u\|_{C^{0,1}(\overline{B}_9)}$, $\min_{\overline{B}_9}f $ and $\|f\|_{C^{1,1}(\overline{B}_9)}$. It may change from line to line.

Together with Lemma \ref{mean}, we have
\begin{align}\label{p-1}
b(0)\leq  C\int_{B_2} b\sigma_1 dx-\frac{C}{2}\sum_{i,j}\int_{B_2}\varphi b_i\sigma_2^{ij}u_j dx+C.
\end{align}

For each $x_0\in B_2$, without loss of generality, we may assume $D^2u$ is diagonalized at $x_0$. By the fact that $F^{ii}=\frac{f(\sigma_1-\lambda_i)}{\lambda_i\sigma_1}$, we have
\begin{align*}
\sum_i b_i\sigma_2^{ii}=\sum_i b_i(\sigma_1-\lambda_i)=\frac{\sigma_1}{f}\sum_i b_i\lambda_i F^{ii}.
\end{align*}

Define
\begin{align*}
H^{ij}=\sigma_1F^{ij}=\sigma_3^{ij}-\frac{\sigma_3}{\sigma_1}\delta_{ij}=\sigma_3^{ij}-f\delta_{ij}.
\end{align*}

Then
\begin{align*}
\sum_{i,j}|\varphi b_i\sigma_2^{ij}u_j|=&\ \sum_i|\varphi b_i\sigma_2^{ii}u_i|\leq  C \sum_i |b_i|\lambda_i H^{ii} \\
\leq&\  C\sum_{i}H^{ii}b_i^2+C\sum_i H^{ii}\lambda_i^2\\
=&\ C\sum_{i}H^{ii}b_i^2+C\left(\sigma_3\sum_i\lambda_i-\frac{\sigma_3}{\sigma_1}\left(\sigma_1^2-2\sigma_2\right)\right)\\
\leq&\ C\sum_{i}H^{ii}b_i^2+C\sigma_2.
\end{align*}

Plugging into (\ref{p-1}), we have
\begin{align}\label{p-2}
b(0)\leq C\int_{B_2} b\sigma_1 dx+C\sum_{i}\int_{B_2}H^{ii}b_i^2dx+C\int_{B_2}\sigma_2dx+C.
\end{align}

Let $\phi$ be a cutoff function such that $\phi=1$ on $B_2$ and $\phi=0$ outside $B_3$. Then
\begin{align*}
\int_{B_2}\sigma_2dx\leq \int_{B_3}\phi \sigma_2dx=-\frac{1}{2}\sum_{i,j}\int_{B_3}\phi_i \sigma_2^{ij}u_jdx\leq C\int_{B_3}\sigma_1dx\leq C.
\end{align*}

On the other hand,
\begin{align*}
\int_{B_2} b\sigma_1 dx\leq&\ \int_{B_3}\phi b\sigma_1dx=-\sum_i\int_{B_3}(\phi_i b+\phi b_i)u_idx\\
\leq&\ C\int_{B_3}\sigma_1dx+C\int_{B_3}|Db|dx\\
\leq&\ C+C\int_{B_3}|Db|dx.
\end{align*}

Plugging the above two inequalities into (\ref{p-2}), we have
\begin{align}\label{p-3}
b(0)\leq C\int_{B_3}|Db|dx+C\sum_{i}\int_{B_2}H^{ii}b_i^2dx+C.
\end{align}

Now
\begin{align*}
|Db|\leq \sum_i|b_i|\leq \frac{1}{2}\sum_i H^{ii}b_i^2+\frac{1}{2}\sum_i \frac{1}{H^{ii}}.
\end{align*}

By Lemma \ref{lambda_3} and the fact that $H^{11}=\frac{f(\lambda_2+\lambda_3)}{\lambda_1}$, we have
\begin{align*}
\int_{B_3} \frac{1}{H^{11}}dx&=\int_{B_3} \frac{\lambda_1}{f(\lambda_2+\lambda_3)}dx\leq C\int_{B_3}\lambda_1dx\leq C\int_{B_3}\sigma_1dx\leq C.
\end{align*}

Thus
\begin{align*}
\int_{B_3} \frac{1}{H^{33}}dx\leq \int_{B_3} \frac{1}{H^{22}}dx\leq \int_{B_3} \frac{1}{H^{11}}dx\leq C.
\end{align*}

Combining the above three inequalities and plugging into (\ref{p-3}), we have
\begin{align}\label{p-4}
b(0)\leq C\sum_{i}\int_{B_3}H^{ii}b_i^2dx+C.
\end{align}

Since the maximum of two subsolutions is still a subsolution, it follow from Lemma \ref{Jacobi} that $b=\ln \max\{\lambda_1,\Lambda(f)\}$ satisfies
\begin{align*}
\sum_{i,j} H^{ij}b_{ij}\geq \frac{1}{432}\sum_{i,j} H^{ij}b_ib_j -C\sigma_1,
\end{align*}
in the viscosity sense.

By Theorem 1 in \cite{I}, $b$ also satisfies the above inequality in the distribution sense. Let $\Phi$ be a cutoff function such that $\Phi=1$ on $B_3$ and $\Phi=0$ outside $B_4$. Then
\begin{align*}
&\ \sum_{i,j}\int_{B_3}H^{ij}b_ib_jdx\\
\leq&\  \sum_{i,j} \int_{B_4}\Phi^2H^{ij}b_ib_jdx\leq 432\sum_{i,j} \int_{B_4}\Phi^2\left( H^{ij}b_{ij}+C\sigma_1\right) dx\\
\leq &\ -864 \sum_{i,j} \int_{B_4}\Phi\Phi_j H^{ij}b_i dx-C\sum_{i,j}\int_{B_4}\Phi^2 (H^{ij})_j b_idx+C\\
\leq &\ \frac{1}{2}\sum_{i}\int_{B_4} \Phi^2H^{ii}b_i^2dx+C\sum_i \int_{B_4}\Phi_i^2H^{ii}dx+C\sum_i\int_{B_4}\Phi^2 f_ib_idx+C.
\end{align*}
We have used the fact that $\sum_j (H^{ij})_j=\sum_j\left((\sigma_3^{ij})_j-(f\delta_{ij})_j\right)=-f_i$ in the last line.

Together with the fact that $H^{ii}=\sigma_3^{ii}-f$, we have
\begin{align*}
\frac{1}{2}\sum_{i,j}\int_{B_4} \Phi^2H^{ii}b_i^2dx\leq&\ C\sum_i \int_{B_4}\Phi_i^2H^{ii}dx+C\sum_i\int_{B_4}\Phi^2 f_ib_idx+C\\
\leq &\ C\sum_i \int_{B_4} H^{ii}dx-C\sum_i\int_{B_4}\left(2\Phi\Phi_if_i+\Phi^2 f_{ii}\right)bdx+C\\
\leq &\ C\int_{B_4}\left(\sigma_2-3f\right)dx+C\int_{B_4}\sigma_1dx+C\\
\leq &\ C\int_{B_4}\sigma_2dx+C.
\end{align*}

Consequently,
\begin{align*}
\sum_{i}\int_{B_3}H^{ii}b_i^2dx\leq \sum_i \int_{B_4}\Phi^2H^{ii}b_i^2dx\leq C\int_{B_4}\sigma_2dx+C.
\end{align*}

Plugging into (\ref{p-4}), we have
\begin{align*}
b(0)\leq C\int_{B_4}\sigma_2dx+C.
\end{align*}

Let $\Psi$ be a cutoff function such that $\Psi=1$ on $B_4$ and $\Psi=0$ outside $B_5$. Then
\begin{align*}
\int_{B_4}\sigma_2dx\leq \int_{B_5}\Psi \sigma_2dx=-\frac{1}{2}\sum_{i,j}\int_{B_5}\Psi_i \sigma_2^{ij}u_j dx\leq C\int_{B_5}\sigma_1dx\leq C.
\end{align*}

It follows that 
\begin{align*}
b(0)\leq C.
\end{align*}

The theorem is now proved.

\end{proof}

\end{document}